\documentclass[a4paper,11pt]{article}

\usepackage[top=3.0cm, bottom=3.0cm, inner=3.0cm, outer=3.0cm,
includefoot]{geometry}

\usepackage{verbatim}
\usepackage{caption}
\usepackage{url}
\usepackage{amsmath}
\usepackage{geometry}
\usepackage{amssymb}
\usepackage{amsmath}
\usepackage{graphicx}
\usepackage{amsthm}
\usepackage{bbm}
\usepackage{float}
\usepackage{color,soul}
\usepackage{hyperref}
\usepackage[T1]{fontenc}
\usepackage[utf8]{inputenc}
\usepackage{authblk}
\usepackage{booktabs}
\usepackage{longtable}
\usepackage{enumerate}
\usepackage{tikz}
\usetikzlibrary{arrows}
\usetikzlibrary{decorations.markings}
\newcommand{\eChar}{\begin{enumerate}[(i)]}
\newcommand{\eCharR}{\begin{enumerate}[(a)]}
\newcommand{\eBr}{\begin{enumerate}[(1)]}

\title
{
Curvature and local matchings of conference graphs and extensions
}

\author[1]{Kaizhe Chen}
\author[2]{Shiping Liu}
\author[3]{Heng Zhang}

\affil[1,2,3]{School of Mathematical Sciences, University of Science and Technology of China, Hefei} 
\affil[1]{ckz22000259@mail.ustc.edu.cn}
\affil[2]{spliu@ustc.edu.cn}
\affil[3]{hengz@mail.ustc.edu.cn}
\date{}

\theoremstyle{plain}
\newtheorem{lemma}{Lemma}[section]
\newtheorem{theorem}[lemma]{Theorem}

\newtheorem{corollary}[lemma]{Corollary}

\theoremstyle{definition}

\newtheorem{conjecture}[lemma]{Conjecture}
\newtheorem{definition}[lemma]{Definition}

\newtheorem{manualtheoreminner}{Theorem}
\newenvironment{manualtheorem}[1]{%
  \renewcommand\themanualtheoreminner{#1}%
  \manualtheoreminner
}{\endmanualtheoreminner}

\newtheorem{example}[lemma]{Example}

\newtheorem{remark}[lemma]{Remark}

\numberwithin{equation}{section}

\begin{document}

\maketitle

\begin{abstract}
We prove a conjecture of Bonini et al. on the precise values of the Lin--Lu--Yau curvature of conference graphs, i.e., strongly regular graphs with parameters $(4\gamma+1,2\gamma,\gamma-1,\gamma)$.
Our method depends only on the parameter relations and applies to broader classes of amply regular graphs. In particular, we develop a new combinatorial approach to show the existence of local perfect matchings. A key observation is that counting common neighbors leads to useful quadratic polynomials. As a corollary, we derive an interesting number-theoretic result concerning quadratic residues.
\end{abstract}

\section{Introduction}
The study of discrete notions of Ricci curvature has attracted much attention and found important applications in graph theory, see, e.g., \cite{BJL12, BM15, CY96, JL14, LLY11, LY10, Ollivier09, Ollivier10}. Ricci curvature is a key concept in Riemannian geometry and geometric analysis. The introduction of analogous concepts of Ricci curvature on graphs brings new deep insights and tools to the study of graph theory.  In \cite{Ollivier09,Ollivier10}, Ollivier introduced the concept of coarse Ricci curvature for Markov chains on metric spaces, including graphs. In \cite{LLY11}, Lin, Lu, and Yau introduced a modified version of Ollivier's coarse Ricci curvature on graphs. This modified curvature, which we call the Lin--Lu--Yau curvature, has proved very useful in detecting geometric and analytic properties of the underlying graphs \cite{CKKLMP23,CKKLP20,CKKLMP20,M23, MS23, MW19, Salez}.

Explicit calculations of Lin--Lu--Yau curvature on graphs with local regularity restrictions have important applications in the study of locally correctable and some locally testable binary linear codes \cite{IS20}, and in bounding the diameter and eigenvalues of amply regular graphs \cite{HLX24}. A $d$-regular graph $G$ with $n$ vertices is called an {\it{amply regular graph}} with parameters $(n,d,\alpha,\beta)$ if any two adjacent vertices have $\alpha$ common neighbors and any two vertices at distance $2$ have $\beta$ common neighbors. Amply regular graphs with diameter $2$ are called strongly regular graphs. 
Denote by $\kappa(x,y)$ the Lin--Lu--Yau curvature of an edge $xy$. It is shown \cite{HLX24} that, for any edge $xy$ of an amply regular graph $G$ with parameters $(n,d,\alpha, \beta)$ with $\beta>\alpha\geq 1$, we have $\kappa(x,y)\geq \frac{3}{d}$. 

On the other hand, the Lin--Lu--Yau curvature has a natural upper bound. Achieving such an upper bound is equivalent to the existence of certain local perfect matchings. This is formulated precisely as below (see, e.g., \cite[Proposition 2.7]{CKKLMP20}).
\begin{theorem}\cite{CKKLMP20}\label{thm:upperk}
   Let $G=(V,E)$ be a $d$-regular graph. For any edge $xy\in E$, we have
   \[\kappa(x,y)\leq \frac{2+|\Delta_{xy}|}{d},\]
   where $\Delta_{xy}:=\Gamma(x)\cap \Gamma(y)$, $\Gamma(x):=\{z\in V| xz\in E\}$, and $\Gamma(y):=\{z\in V| yz\in E\}$.  Moreover, the following two propositions are equivalent:
   \begin{itemize}
   \item[(a)] $\kappa(x,y) = \frac{2+|\Delta_{xy}|}{d}$;
   \item[(b)] there is a perfect matching between  $N_x$ and $N_y$, where
     $N_x:= \Gamma(x) \backslash (\Delta_{xy} \cup \{y\}) $ and
     $N_y:= \Gamma(y) \backslash (\Delta_{xy} \cup \{x\}) $.
   \end{itemize}
\end{theorem}
The crucial role of local matching conditions in the curvature calculation was first observed in \cite{BM15, Smith}, and further developed in \cite{BCDDFP20,DOJ19,MW19}. For an amply regular graph  with parameters $(n,d,\alpha,\beta)$, the above upper bound reads $\kappa(x,y)\leq \frac{2+\alpha}{d}$. A natural question is for which parameters this upper bound is achieved. 

Bonini et al. \cite{BCDDFP20} derived formulas for the Lin--Lu--Yau curvature of strongly regular graphs in terms of the graph parameters and the size of a maximal matching in the core neighborhood. They further proposed the following conjecture \cite[Conjecture 1.7]{BCDDFP20}. 
\begin{conjecture}\cite{BCDDFP20}\label{conj:Bonini}
    Let $G=(V,E)$ be a strongly regular conference graph with parameters $(4\gamma+1, 2\gamma, \gamma-1,\gamma)$ with $\gamma\geq 2$. Then the Lin--Lu--Yau curvature satisfies 
    \[\kappa(x,y)=\frac{1}{2}+\frac{1}{2\gamma}, \,\,\text{for all}\,\,xy\in E.\]
\end{conjecture}

Recall that a {\it conference graph} is a strongly regular graph with parameters $(4\gamma+1, 2\gamma, \gamma-1,\gamma)$. Conference graphs have very interesting properties: While being highly symmetric, they are quasirandom \cite{AS16}. The existence of conference graphs is known for each $\gamma$ such that $4\gamma+1$ is a prime power (i.e. Paley graphs). Let $n=4\gamma+1$ be a prime power and $GF(n)$ be the unique finite field of order $n$. Recall that a {\it Paley graph} $P(n)$ of order $n$ is a graph on $n$ vertices with two vertices adjacent iff their difference is a non-zero square in $GF(n)$. 
By Theorem \ref{thm:upperk}, the Lin--Lu--Yau curvature of a conference graph is no greater than $\frac{2+(\gamma-1)}{2\gamma}=\frac{1}{2}+\frac{1}{2\gamma}$. Therefore, Conjecture \ref{conj:Bonini} in fact claims that the upper bound is achieved for conference graphs. 

We stress that, in general, the Lin--Lu--Yau curvature of an amply regular graph is not necessarily uniquely determined by its parameters. 
For instance, the $4 \times 4$ Rook's graph and the Shrikhande graph are both strongly regular with parameters $(16,6,2,2)$. However, Bonini et al. \cite{BCDDFP20} computed their Lin--Lu--Yau curvatures to be $\frac{2}{3}$ and $\frac{1}{3}$, respectively. In particular, the $4 \times 4$ Rook's graph attains the upper bound of the Lin--Lu--Yau curvature, while the Shrikhande graph does not.

Huang, Liu, and Xia showed \cite[Proposition 3.6]{HLX24} that $\kappa(x,y)=\frac{2+\alpha}{d}$ for any amply regular graph with parameters $(n,d,\alpha,\beta)$ such that $d\leq 2\beta-\alpha-1$. However, the parameters of conference graphs, namely $(4\gamma+1, 2\gamma, \gamma-1, \gamma)$, do not satisfy this condition. 

A recent work by Bonini, Liu, Semmens, and Tran \cite{BCCST24} proves Conjecture \ref{conj:Bonini} for Paley graphs $P(n)$ of even power order $n=p^{2m}$. 
Very interesting extensions to generalized Paley graphs have also been carried out in \cite{BCCST24}.

In this paper, we provide several sufficient conditions for the curvature of an amply regular graph to reach the natural upper bound. As a corollary, we prove Conjecture \ref{conj:Bonini}.
\begin{theorem}\label{thm2}
    Let $G=(V,E)$ be an amply regular graph with parameters $(n,d,\alpha,\beta)$ such that $n<3d-2\alpha$. Suppose that one of the following conditions holds:
    \begin{itemize}
      \item[(1)] $\beta=\alpha+1$ and $d> \alpha+\sqrt{6\alpha+\frac{1}{4}}+\frac{3}{2}$;
      \item[(2)] $\beta\ge\alpha+2$ and $d \ge \frac{3}{2}\beta + \frac{1}{2}\sqrt{4\alpha^2-3\beta^2+28\beta}$;
    \item[(3)] $\beta \geq \frac{2\sqrt{3}}{3}\alpha + 7$.
    \end{itemize}
    Then the Lin--Lu--Yau curvature satisfies \[\kappa(x,y)=\frac{2+\alpha}{d}, \,\,\text{for all}\,\,xy\in E.\]
\end{theorem}
In particular, condition (1) of Theorem \ref{thm2} covers all conference graphs with $\gamma > 6$. The cases for $2 \leq \gamma \leq 6$ are verified separately in Section \ref{section3}.
\begin{theorem}\label{Bonini}
The Conjecture \ref{conj:Bonini} holds true.
\end{theorem}
    By Theorem \ref{thm:upperk}, proving Theorem \ref{thm2} reduces to establishing the following result in terms of local matchings, which is equivalent to Theorem \ref{thm2}.
\begin{manualtheorem}{\ref{thm2}'}
     Let $G=(V,E)$ be an amply regular graph with parameters $(n,d,\alpha,\beta)$. If one of the $3$ parameter conditions in Theorem \ref{thm2} holds, then there exists a perfect matching between $N_x$ and $N_y$ for every $xy\in E$.
\end{manualtheorem}


Besides conference graphs, our parameter conditions are also fulfilled by many other strongly regular graphs, as shown in the arXiv version of this paper (see \cite[Appendix]{CLZ24}). In fact, we provide five more refined conditions in \cite[Theorem 1.4]{CLZ24}, which generalize the three conditions in Theorem \ref{thm2}. For brevity, we do not include them here.

Our proof provides a new combinatorial method for showing the existence of local perfect matchings, which is also applicable to many cases of amply regular graphs beyond those in Theorem \ref{thm2} and \cite[Theorem 1.4]{CLZ24}. For example, the restriction $n<3d-2\alpha$ is not always necessary. We refer to Remark \ref{rmk4.2} and Example \ref{ex:324} for more detailed discussions. 

Applying Theorem \ref{Bonini} to Paley graphs yields an interesting number-theoretic result. Theorem \ref{Bonini} tells that there exists a perfect matching between $N_x$ and $N_y$ for any adjacent vertices $x,y$ of a Paley graph $P(n)$ with $n>5$. This leads to the following corollary.
\begin{corollary}\label{cor:number_theoretic}
Let $n>5$ be a prime power congruent to $1$ modulo $4$. Consider two elements $x, y \in GF(n)$ such that $x - y$ is a non-zero square in $GF(n)$. For any subset $S \subset GF(n) \setminus \{x, y\}$ with $|S| \geq \frac{3}{4}(n - 1)$, there exist $w, z \in S$ such that $x - w, w - z, z - y$ are non-zero squares and $x - z, y - w$ are not squares in $GF(n)$.
\end{corollary}
In particular, the above corollary applies to quadratic residues when $n$ is a Pythagorean prime (i.e., a prime congruent to $1$ modulo $4$) greater than $5$.

The rest of this paper is organized as follows.
In Section \ref{section2}, we recall definitions and properties about the Lin--Lu--Yau curvature, perfect matching, and strongly/amply regular graphs. In Section \ref{section3}, we prove Theorem \ref{thm2}, Theorem \ref{Bonini}, and Corollary \ref{cor:number_theoretic}, and present possible ways to generalize our conclusion.

   

   \section{Preliminaries}\label{section2}
   \subsection{Lin--Lu--Yau curvature and local matchings}\label{sec1}
   \begin{definition}[Wasserstein distance]
   
     Let $G=(V,E)$ be a locally finite graph, $\mu_1$ and $\mu_2$ be two probability measures on $G$. The Wasserstein distance $W_1(\mu_1, \mu_2)$ between $\mu_1$ and $\mu_2$ is defined as
     \[W_1(\mu_1,\mu_2)=\inf_{\pi}\sum_{y\in V}\sum_{x\in V}d(x,y)\pi(x,y),\]
     where $d(x,y)$ denotes the combinatorial distance between $x$ and $y$ in $G$, and the infimum is taken over all maps $\pi: V\times V\to [0,1]$ satisfying
     $$m_1(x)=\sum\limits_{y\in V}\pi(x,y) \text{ for any $x\in V$ and }m_2(y)=\sum\limits_{x\in V}\pi(x,y) \text{ for any $y\in V$}.$$
     Such a map is called a transport plan.
     \end{definition}
     We consider the following particular measure around a vertex $x\in V$:
     \[\mu_x^p(y)=\left\{
                    \begin{array}{ll}
                      p, & \hbox{if $y=x$;} \\
                      \frac{1-p}{\mathrm{deg}(x)}, & \hbox{if $yx\in E$;} \\
                      0, & \hbox{otherwise,}
                    \end{array}
                  \right.
     \]
     where $\mathrm{deg}(x)$ is the degree of the vertex $x$.

     \begin{definition}[$p$-Ollivier curvature \cite{Ollivier09} and Lin--Lu--Yau curvature \cite{LLY11}] Let $G=(V,E)$ be a locally finite graph. For any vertices $x,y\in V$, the $p$-Ollivier curvature $\kappa_p(x,y)$, $p\in [0,1]$, is defined as
       \[\kappa_p(x,y)=1-\frac{W_1(\mu_x^p,\mu_y^p)}{d(x,y)}.\]
       The Lin--Lu--Yau curvature $\kappa(x,y)$ is defined as
       \[\kappa(x,y)=\lim_{p\to 1}\frac{\kappa_p(x,y)}{1-p}.\]
       \end{definition}
       Notice that $\kappa_1(x,y)$ is always $0$. Hence, the Lin--Lu--Yau curvature $\kappa(x,y)$ is equal to the negative of the left derivative of the function $p\mapsto \kappa_p(x,y)$ at $p=1$.

Bourne et al. \cite{BCLMP18} proved the following relation between $p$-Ollivier curvature and Lin--Lu--Yau curvature for an edge $xy$ of a $d$-regular graph $G$:
\begin{equation}\label{BCLMP}
    \kappa(x,y)=\frac{d+1}{d}\kappa_{\frac{1}{d+1}}(x,y).
\end{equation}
By the property of the Monge problem, this relation leads in particular to the following expression of Lin--Lu--Yau curvature of a $d$-regular graph:
\begin{equation}\label{def2}
    \kappa(x,y)=\frac{1}{d}\left(d+1-\min_{\substack{\phi: N_x\to N_y \\\text{bijective}}}\sum_{v\in N_x}d(v,\phi(v))\right).
\end{equation}
Theorem \ref{thm:upperk} on the upper bound of Lin--Lu--Yau curvature and its relation to existence of a perfect matching between $N_x$ and $N_y$ follows then directly from \eqref{def2}. For more details, we refer to \cite[Proposition 2.7]{CKKLMP20}.
   
     Let us recall the definition of a \emph{perfect matching}.

       \begin{definition}\cite[Section 16.1]{BM08}
        Let $G=(V,E)$ be a locally finite simple connected graph. A set $M$ of pairwise non-adjacent edges is called a \emph{matching}. 
        Each vertex adjacent to an edge of $M$ is said to be \emph{covered} by $M$. 
        A matching $M$ is called a \emph{perfect matching} if it covers every vertex of the graph.
        \end{definition}
 Hall's marriage Theorem is a fundamental tool for proving Conjecture \ref{conj:Bonini}. 
   \begin{theorem}\cite[Theorem 16.4, Hall's Marriage Theorem]{BM08}\label{lemma:Hall}
   Let $H=(V,E)$ be a bipartite graph with the bipartition $V=V_1\sqcup V_2$. Then $H$ has a perfect matching  if and only if
   \[|V_1|=|V_2|\,\, \text{and}\,\, |\Gamma_{V_2}(S)|\geq |S| \,\,\text{for all}\,\,S\subseteq V_1,\]
   where $\Gamma_{V_2}(S):=\{v\in V_2 |\,\,\text{there exists }\, w\in S \,\text{such that}\,\, vw\in E\}$.
    \end{theorem}

\subsection{Strongly/Amply regular graph }\label{sec2}
\begin{definition}[Strongly/Amply regular graph \cite{BCN89}] Let $G$ be a $d$-regular graph with $n$ vertices which is neither complete nor empty.

\noindent The graph $G$ is called a strongly regular graph with parameters $(n,d,\alpha,\beta)$ if:
        \begin{itemize}
           \item [(i)] Any two adjacent vertices have $\alpha$ common neighbors;
           \item [(ii)] Any two non-adjacent vertices have $\beta$ common neighbors.
        \end{itemize}
    The graph $G$ is called an amply regular graph with parameters $(n, d,$ $\alpha, \beta)$ if:
        \begin{itemize}
            \item [(i)] Any two adjacent vertices have $\alpha$ common neighbors;
            \item [(ii)] Any two vertices at distance two have $\beta$ common neighbors.
        \end{itemize}
\end{definition}

The following theorem shows a basic relationship between the parameters of an amply regular graph.

            \begin{theorem}\label{xs}
       Let $G=(V, E)$ be an amply regular graph with parameters $(n, d, \alpha, \beta)$. Then,
       \begin{equation}\label{eq:arg_ineq}
           d(d-\alpha-1)\le (n-d-1)\beta.
       \end{equation}
    Moreover, the equality holds if and only if $G$ is strongly regular.
   \end{theorem}
   \begin{proof}
       Let $M$ be the number of pairs of $(\{v_1, v_2\}, u)$, where $v_1, v_2, u\in V$ and $u$ is a common neighbour of $v_1$ and $v_2$. For any $u\in V$, there are $\binom{d}{2}$ pairs of such $\{v_1, v_2\}$. Thus, $M=\binom{d}{2}n$. For any two adjacent vertices $v_1$ and $v_2$, there are $\alpha$ such $u$. For any two non-adjacent vertices $v_1$ and $v_2$, there are at either $\beta$ or $0$ such $u$. Therefore,$$M\le e\alpha+\left(\binom{n}{2}-e\right)\beta,$$where $e=|E|=\frac{nd}{2}$, and the inequality \eqref{eq:arg_ineq} follows. It is clear that the equality holds if and only if $G$ is strongly regular.
   \end{proof}

  \section{Proof of Theorem \ref{thm2}, Theorem \ref{Bonini}, and Corollary \ref{cor:number_theoretic}}\label{section3}
In this section, we prove Theorem \ref{thm2}, Theorem \ref{Bonini}, and Corollary \ref{cor:number_theoretic}. First, we observe the following lemma which is useful in applying Hall's marriage Theorem  (Theorem \ref{lemma:Hall}). 

\begin{lemma}\label{lemma:Hall_reduction}
     Let $H=(V,E)$ be a bipartite graph with the bipartition $V=V_1\sqcup V_2$ such that $|V_1|=|V_2|$. Denote by $m:=|V_1|=|V_2|$. Suppose that 
     \begin{itemize}
         \item [(i)] $|\Gamma_{V_2}(S)|\geq |S|$
     for any $S\subseteq V_1$ with $|S|\leq \frac{m+1}{2}$;
     \item [(ii)] $|\Gamma_{V_1}(S)|\geq |S|$
     for any $S\subseteq V_2$ with $|S|\leq \frac{m+1}{2}$.
     \end{itemize}
     Then, we have $|\Gamma_{V_2}(S)|\geq |S|$ for any $S\subset V_1$ and  $|\Gamma_{V_1}(S)|\geq |S|$ for any $S\subset V_2$. 
\end{lemma}
\begin{proof}
   Let $S\subseteq V_1$ be a subset of $V_1$ such that $|S|>\frac{m+1}{2}$. Pick a subset $S_1\subset S$ with $|S_1|=\left\lfloor \frac{m+1}{2}\right\rfloor$. It follows from the assumption (i) that 
   \[|\Gamma_{V_2}(S)|\geq |\Gamma_{V_2}(S_1)|\geq |S_1|=\left\lfloor \frac{m+1}{2}\right\rfloor.\]
This implies that
\[|V_2\setminus \Gamma_{V_2}(S)|=|V_2|-|\Gamma_{V_2}(S)|\leq m-\left\lfloor \frac{m+1}{2}\right\rfloor\leq \frac{m}{2}.\]
Due to the assumption (ii), we have
\[|\Gamma_{V_1}(V_2\setminus\Gamma_{V_2}(S))|\geq |V_2\setminus\Gamma_{V_2}(S)|.\]
By definition, we have $\Gamma_{V_1}(V_2\setminus\Gamma_{V_2}(S))\cap S=\emptyset$. Thus, we estimate
\[m=|V_1|\geq |\Gamma_{V_1}(V_2\setminus\Gamma_{V_2}(S))|+|S|\geq |V_2\setminus \Gamma_{V_2}(S)|+|S|=m-|\Gamma_{V_2}(S)|+|S|.\]
Hence we have $|\Gamma_{V_2}(S)|\geq |S|$. This proves 
$|\Gamma_{V_2}(S)|\geq |S|$ for any $S\subseteq V_1$. 
By symmetry, we have $|\Gamma_{V_1}(S)|\geq |S|$ for any $S\subseteq V_2$. 
\end{proof}

The following lemma is a direct consequence of Theorem \ref{xs}.

\begin{lemma}\label{lemm:size}
    Let $G=(V,E)$ be an amply regular graphs with parameters $(n,d,\alpha,\beta)$ such that $d\geq \alpha+2$. If $n<3d-2\alpha$, then $d\leq 2\beta$.
\end{lemma}
\begin{proof}
    If $n<3d-2\alpha$, we derive from Theorem \ref{xs} that 
    \begin{equation}\label{eq:arg}
        d(d-\alpha-1)\le (n-d-1)\beta\le2(d-\alpha-1)\beta.
    \end{equation}
    Since $d\geq \alpha+2$, we deduce $d\le 2\beta$. 
\end{proof}


Let $x,y$ be two adjacent vertices of an amply regular graph $G=(V,E)$ with parameters $(n,d,\alpha,\beta)$. 
We decompose the vertex set $V$ into $6$ disjoint subsets as follows (depicted in Figure \ref{fig:1}):
    $$V=\{x\}\cup\{y\}\cup\Delta_{xy}\cup N_x\cup N_y\cup P_{xy}.$$  
    
\begin{figure}[ht]
    \begin{center}
    \begin{tikzpicture}[scale=0.5] 
        \node (x) at (2,4) {$x$};   
        \node (y) at (8,4) {$y$};

        \draw (2,0) ellipse (1 and 0.5) node {$N_y$}; 
        \draw (8,0) ellipse (1 and 0.5) node {$N_x$}; 
        \draw (5,0) ellipse (1 and 0.5) node {$\Delta_{xy}$};

        \draw (5,-2) ellipse (0.7 and 0.7) node {$P_{xy}$};

        \draw (x) -- (1.5, 0.4);  
        \draw (x) -- (4.5, 0.4);
        \draw (y) -- (5.5, 0.4);
        \draw (y) -- (8.5, 0.4);

        \draw (x) -- (y);  
    \end{tikzpicture}
    \caption{A schematic plot for the decomposition of $V$}\label{fig:1}
\end{center}
\end{figure}
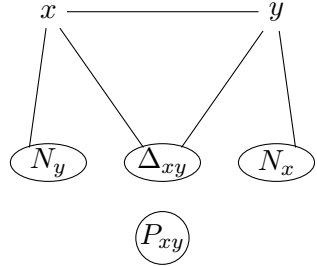
The subsets above are defined by $\Delta_{xy}:=\Gamma(x)\cap \Gamma(y)$,
$N_x:= \Gamma(x) \backslash (\Delta_{xy} \cup \{y\})$,
$N_y:= \Gamma(y) \backslash (\Delta_{xy} \cup \{x\})$, and
$P_{xy}:=V\backslash(\Gamma(x)\cup\Gamma(y))$, where $\Gamma(x):=\{z\in V|\, xz\in E\}$ and $\Gamma(y):=\{z\in V| \,yz\in E\}$. 
It is direct to check by definition that
    \begin{equation}\label{eq:counting}
        |\Delta_{xy}|=\alpha,\, |N_x|=|N_y|=d-\alpha-1,\, |P_{xy}|=n-2d+\alpha.
    \end{equation}
Next, we prepare the following lemma.
\begin{lemma}\label{lemma:neighbors}
    Let $G=(V,E)$ be an amply regular graphs with parameters $(n,d,\alpha,\beta)$ and $x,y$ be two adjacent vertices. Suppose that a vertex $v\in N_x$ has $\ell$ neighbors in $N_y$. Then,
    \begin{itemize}
        \item [(i)] $v$ has $\beta-1-\ell$ neighbors in $\Delta_{xy}$;
        \item [(ii)]  $v$ has $\alpha-\beta+1+\ell$ neighbors in $N_x$;
         \item [(ii)]  $v$ has $d-\alpha-1-\ell$ neighbors in $P_{xy}$.
    \end{itemize}
\end{lemma}
\begin{proof}
 Since $y$ and $v$ have $\beta$ common  neighbors, $v$ has $\beta-1-\ell$ neighbors in $\Delta_{xy}$. Since $x$ and $v$ have $\alpha$ common neighbors, $v$ has $\alpha-\beta+1+\ell$ neighbors in $N_x$. Since $v$  has $d-1$ neighbors except $x$, the number of neighbors of $v$ in $P_{xy}$ is $d-\alpha-1-\ell$.   
\end{proof}

We now proceed to the proof of Theorem \ref{thm2}.
\begin{proof}[Proof of Theorem \ref{thm2}]
  Let $G=(V,E)$ be an amply regular graph satisfying one of the conditions (1)--(3), and let $xy\in E$ be an edge. Recall from \cite[Proposition 3.6]{HLX24} that $\kappa(x,y)=\frac{2+\alpha}{d}$ holds whenever $d\leq 2\beta-\alpha-1$. Thus, we assume throughout the proof that $d\ge 2\beta-\alpha$. 

  We claim that each of the conditions (1)--(3) yields $d>\alpha+2$. Indeed, it is clear that condition (1)  yields $d>\alpha+2$. Under the condition (2), $d\ge\frac{3}{2}\beta>\alpha+2$. Since $d\ge \beta$, condition (3) implies $d>\alpha+2$.
  Therefore, the claim holds. Recalling  \eqref{eq:counting}, we derive $|N_x|=|N_y|=d-\alpha-1>1$. That is, both $N_x$ and $N_y$ contain at least two vertices.

  Recall from \cite[Theorem 1.5]{LL21} that $\kappa(x,y)=\frac{2+\alpha}{d}$ holds whenever $\alpha=0$ and $\beta\geq 2$. If $\alpha=0$, then the conditions (2) and (3) both lead to $\beta\geq 2$ and hence $\kappa(x,y)=\frac{2}{d}$. Under the condition (1), $\alpha=0$ implies $\beta=1$ and $d>2$. However, by the condition $n<3d-2\alpha$ and Lemma \ref{lemm:size}, we have $d\leq 2\beta=2$, which is a contradiction. That is, under the condition (1), the case $\alpha=0$ is not allowed. 
  In the following proof, we assume that $\alpha>0$.

  By Theorem \ref{thm:upperk}, we only need to prove that there is a perfect matching between $N_x$ and $N_y$. By Hall's marriage Theorem (see Theorem \ref{lemma:Hall}), it remains to prove that, for any subset $S\subseteq N_x$, the number $|\Gamma_{N_y}(S)|$ of neighbors in $N_y$ is at least $|S|$. For simplicity, we denote $S:=\{v_1,v_2,\ldots,v_b\}$ with $b=|S|$ and $T:=\Gamma_{N_y}(S)$. We need to show the inequality
  \begin{equation}\label{eq:goal}
      |T|\geq b
  \end{equation}
  always holds.
For each vertex $v_i \in S$, let $x_i$ be the number of neighbors of $v_i$ in $N_y$. 
By Lemma \ref{lemma:neighbors}, the vertex $v_i$ has $d-\alpha-1-x_i$ neighbors in $P_{xy}$.
If $|P_{xy}|=0$, then we have $x_i=d-\alpha-1=|N_y|$. Therefore, the inequality \eqref{eq:goal} clearly holds. In the following proof, we assume that $|P_{xy}|>0$, i.e., $n-2d+\alpha>0$.

By symmetry and Lemma \ref{lemma:Hall_reduction}, we only need to show the inequality \eqref{eq:goal} for $b\leq \frac{d-\alpha}{2}$.

Assume, for the sake of contradiction, that $|T|\le b-1$. 

    \underline{Case 1:}  $b=1$. 

    If $b=1$, then $v_1 $ has no neighbors in $N_y$, that is, $x_1=0$. Thus, by Lemma \ref{lemma:neighbors}, $v_1$ has $\beta-1$ neighbors in $\Delta_{xy}$ and $d-\alpha-1$ neighbors in $P_{xy}$. If condition (2) or (3) holds, then $\beta-1>\alpha= |\Delta_{xy}|$, which is a contradiction. 
    If condition (1) holds,
    we have $\beta-1= \alpha$ and hence $v_1$ is adjacent to all vertices in $\Delta_{xy}$. Since $n<3d-2\alpha$, we have $|P_{xy}|=n-2d+\alpha\le d-\alpha-1$. Thus, $|P_{xy}|=d-\alpha-1$ and $v_1$ is adjacent to all vertices in $P_{xy}$. Therefore, $v_1$ has in total $d-1$ neighbors in $\Delta_{xy}\cup P_{xy}$. Hence, $v_1$ has no neighbors in $N_x\cup N_y$.

     We pick any vertex  $v \in N_x\setminus\{v_1\}$, and suppose that $v$ has $p$ neighbors in $N_y$. Then $v$ has $\beta-p-1$ neighbors in $\Delta_{xy}$ and $d-\alpha-1-p$ neighbors in $P_{xy}$. Therefore, $v$ and $v_1$ have a total of $(\beta-1-p)+(d-\alpha-1-p)+1$ common neighbors. Since $v$ and $v_1$ are of distance $2$, the number of common neighbors of $v$ and $v_1$ equals to $\beta$. Thus,
    \begin{align*}
      (\beta-1-p)+(d-\alpha-1-p)+1=\beta,
    \end{align*}
    that is,
     \begin{align}\label{2p}
      2p=d-\alpha-1,
    \end{align}
    which implies that $d-\alpha-1$ is even.

    Similarly, we pick any vertex  $u \in N_y$, and suppose that $u$ has $q$ neighbors in $N_x$. Then $u$ has $\beta-q-1$ neighbors in $\Delta_{xy}$ and $d-\alpha-1-q$ neighbors in $P_{xy}$. Therefore, $u$ and $v_1$ have a total of $(\beta-q-1)+(d-\alpha-1-q)$ common neighbors. Since $u$ and $v_1$ are non-adjacent, we have $q<|N_x|=d-\alpha-1$. Thus, $u$ and $v_1$ have at least one common neighbor, and hence the number of common neighbors of $u$ and $v_1$ equals to $\beta$. Then,
    \begin{align*}
      (\beta-q-1)+(d-\alpha-1-q)=\beta,
    \end{align*}
  that is,\begin{align}\label{2q}
    2q=d-\alpha-2.
  \end{align}
  Thus, $d-\alpha-1$ is odd, contradicting \eqref{2p}.

  \underline{Case 2}: $2\le b\le\frac{d-\alpha}{2}$.

 Let $X=x_1+\cdots+x_b$, then $X$ is the number of edges between $S$ and $T$.
 
     Set $\Delta_{xy}:=\{\delta_1,\ldots,\delta_{\alpha}\}$, and let $A_k$ be the number of neighbors of $\delta_k$ in $S$, where $1\le k \le \alpha$. Counting the number of edges between $S$ and $\Delta_{xy}$, we have 
     $$\sum_{k=1}^{\alpha} A_k =\sum_{i=1}^{b} (\beta-x_i-1) =(\beta-1)b-X.$$
    Let $X_{ij}$ be the number of common neighbors of $v_i$ and $v_j$ in $\Delta_{xy}$, and denote by
    \[M_1:=\sum_{1 \le i<j \le b}X_{ij}.\] 
    Note that, for any $\delta_k \in \Delta_{xy}$, there are $\binom{A_k}{2}$ pairs of $v_i,v_j$ so that $\delta_k$ is a common neighbor of them.
    We have
    $$M_1=\sum_{k=1}^{\alpha} \binom{A_k}{2}=\frac{1}{2}\left[\sum_{k=1}^{\alpha} A_k^2-\sum_{k=1}^{\alpha} A_k\right].$$
    By the Cauchy inequality, we derive
    \begin{align}\label{M1}
      M_1 &\ge \frac{1}{2}\left[\frac{1}{\alpha}\left(\sum_{k=1}^{\alpha} A_k\right)^2-\sum_{k=1}^{\alpha} A_k\right]  \notag \\
       &=\frac{1}{2}\left[\frac{\left((\beta-1)b-X\right)^2}{\alpha}-\left((\beta-1)b-X\right)\right].
    \end{align}

    Set $P_{xy}:=\{\epsilon_1,\ldots,\epsilon_{n-2d+\alpha}\}$, and let $B_k$ be the number of neighbors of $\epsilon_k$ in $S$, where $1\le k \le n-2d+\alpha$. Similarly, we have 
 $$\sum_{k=1}^{n-2d+\alpha} B_k =\sum_{i=1}^{b} (d-\alpha-1-x_i) =(d-\alpha-1) b-X.$$
Let $Y_{ij}$ be the number of common neighbors of $v_i$ and $v_j$ in $P_{xy}$, and denote by \[M_2:=\sum_{1 \le i<j \le b}Y_{ij}.\] Since $n-2d+\alpha\le d-\alpha-1$ under each condition $(1)-(5)$, it follows by Cauchy inequality that
\begin{align}
  \notag  M_2=\sum_{k=1}^{n-2d+\alpha} \binom{B_k}{2}&\ge \frac{1}{2}\left[\frac{1}{n-2d+\alpha}\left(\sum_{k=1}^{n-2d+\alpha} B_k\right)^2-\sum_{k=1}^{n-2d+\alpha} B_k\right]
      \\ \label{m2} &=\frac{1}{2}\left[\frac{\left( (d-\alpha-1) b-X\right)^2}{n-2d+\alpha}-\left( (d-\alpha-1) b-X\right)\right]\\ \label{M2}  &\ge  \frac{1}{2}\left[\frac{\left( (d-\alpha-1) b-X\right)^2}{d-\alpha-1}-\left( (d-\alpha-1) b-X\right)\right].
    \end{align}

     Set $T:=\{u_1,\ldots,u_{|T|}\}$, and let $C_k$ be the number of neighbors of $u_k$ in $S$, where $1\le k \le |T|$. Similarly, we have 
    $$\sum_{k=1}^{|T|} C_k =X.$$
   Let $Z_{ij}$ be the number of common neighbors of $v_i$ and $v_j$ in $T$, and denote by \[M_3:=\sum_{1 \le i<j \le b}Z_{ij}.\] Similarly, we have
       $$M_3=\sum_{k=1}^{|T|} \binom{C_k}{2}=\frac{1}{2}\left[\sum_{k=1}^{|T|} C_k^2-\sum_{k=1}^{|T|} C_k\right].$$
   By the assumption that $|T|\le b-1$, we derive
       \begin{align}\label{M3}
         M_3 \ge \frac{1}{2}\left[\frac{1}{|T|}\left(\sum_{k=1}^{|T|} C_k\right)^2-\sum_{k=1}^{|T|} C_k\right]
          \ge \frac{1}{2}\left[\frac{1}{b-1}X^2-X\right].
       \end{align}

Set $N_x=\{v_1,\ldots,v_{d-\alpha-1}\}$, and let $D_k$ be the number of neighbors of $v_k$ in $S$, where $1\le k \le d-\alpha-1$. 
Counting the number of ordered pairs $(v_k,v_j)$ with $v_k\in N_x$, $v_j\in S$ and $v_kv_j\in E$, we have
       $$\sum_{k=1}^{d-\alpha-1} D_k=\sum_{i=1}^b (\alpha-\beta+1+x_i) =(\alpha-\beta+1)b+X.$$
      Let $W_{ij}$ be the number of common neighbors of $v_i$ and $v_j$ in $N_x$, and denote by \[M_4:=\sum_{1 \le i<j \le b}W_{ij}.\] Similarly, we have
      \begin{align}\label{M4}
      \notag  M_4=\sum_{k=1}^{d-\alpha-1} \binom{D_k}{2}&\ge \frac{1}{2} \left[\frac{1}{d-\alpha-1}\left(\sum_{k=1}^{d-\alpha-1} D_k\right)^2-\sum_{k=1}^{d-\alpha-1} D_k\right]\\
             &=\frac{1}{2}\left[\frac{((\alpha-\beta+1)b+X)^2}{d-\alpha-1}-((\alpha-\beta+1)b+X)\right].
          \end{align}
 Since $v_i$ and $v_j$ have either $\alpha$ or $\beta$ common neighbors, we have \begin{align}\label{leb}\notag
 M_1+M_2+M_3+M_4+\binom{b}{2}&=\sum_{1 \le i<j \le b}(X_{ij}+Y_{ij}+Z_{ij}+W_{ij}+1)\\ &\le \binom{b}{2}\max\{ \alpha,\beta\}.\end{align}
We divide the remaining argument of Case 2 into two subcases.

(i) Suppose that the condition (1) holds. Inserting \eqref{M1}, \eqref{M2}, \eqref{M3}, \eqref{M4} into \eqref{leb} yields
\begin{align}\label{dxs2}
  \left(\frac{1}{\alpha}+\frac{2}{d-\alpha-1}+\frac{1}{b-1}\right)X^2-4bX+(b^2-b)(d-\alpha-1)\le0.
\end{align}
Since the coefficient of $X^2$ is positive, the discriminant of the above quadratic polynomial has to be non-negative.  Therefore, we derive
\begin{align}\label{4b}
  \notag 4b\ge&\left(\frac{1}{\alpha}+\frac{2}{d-\alpha-1}+\frac{1}{b-1}\right)(b-1)(d-\alpha-1)\\
=&2(b-1)+(d-\alpha-1)+\frac{(b-1)(d-\alpha-1)}{\alpha}.
\end{align}That is,
$$d-\alpha-3-\frac{d-\alpha-1}{\alpha}\le \left(2-\frac{d-\alpha-1}{\alpha}\right)b.$$
Since $d\le2\beta\le 3\alpha+1$, it follows that $2-\frac{d-\alpha-1}{\alpha}\ge 0$. By the assumption that $b\le \frac{d-\alpha}{2}$, we have $$ d-\alpha-3-\frac{d-\alpha-1}{\alpha} \le \left(2-\frac{d-\alpha-1}{\alpha}\right)\frac{d-\alpha}{2}. $$
Considering the above inequality as a quadratic inequality with respect to  $d$, we deduce
$$d\le \alpha+\sqrt{6\alpha+\frac{1}{4}}+\frac{3}{2},$$
which is a contradiction.

(ii) Suppose that either condition (2) or (3) holds. Inserting \eqref{M1}, \eqref{M2}, \eqref{M3} and $M_4\ge 0$ into the inequality \eqref{leb} leads to 
\begin{align*}
    \left(\frac{1}{\alpha}+\frac{1}{d-\alpha-1} +\frac{1}{b-1}\right)X^2&-\left(\frac{2b(\beta-1)}{\alpha}+2b-1\right)X\\& +\frac{b^2(\beta-1)^2}{\alpha}-(\beta-1)b+(b^2-b)(d-\alpha-\beta)\le0.
\end{align*}
Using $\frac{2b(\beta-1)}{\alpha}+2b-1<\frac{2b(\beta-1)}{\alpha}+2b$, we derive from the non-negativity of the discriminant of the above quadratic polynomial that
$$b\left(\frac{\beta-1}{\alpha}+1\right)^2\ge \left[\frac{1}{\alpha}+\frac{1}{d-\alpha-1} +\frac{1}{b-1}\right]\left[\frac{b(\beta-1)^2}{\alpha}-(\beta-1)+(b-1)(d-\alpha-\beta)\right].$$
That is,
\begin{align*}&\frac{2b(\beta-1)}{\alpha}+2b\\\ge &1+\frac{bd-b\beta-d+1}{\alpha}+\left[\frac{1}{d-\alpha-1} +\frac{1}{b-1}\right]\left[\frac{b(\beta-1)^2}{\alpha}-(\beta-1)+(b-1)(d-\alpha-\beta)\right].
\end{align*}
Since$$\frac{b(\beta-1)^2}{\alpha}-(\beta-1)+(b-1)(d-\alpha-\beta)=(b-1)(d-\alpha-1)+b(\beta-1)\left(\frac{\beta-1}{\alpha}-1\right),$$
we derive 
\begin{align*}  
\notag&\frac{2b(\beta-1)}{\alpha}+b\\\ge& \frac{bd-b\beta-d+1}{\alpha}+d-\alpha-1+\left(\frac{1}{d-\alpha-1} +\frac{1}{b-1}\right)b(\beta-1)\left(\frac{\beta-1}{\alpha}-1\right)\\ \notag >&\frac{bd-b\beta-d+1}{\alpha}+d-\alpha-1+\frac{b(\beta-1)}{d-\alpha-1}\left(\frac{\beta-1}{\alpha}-1\right)+ (\beta-1)\left(\frac{\beta-1}{\alpha}-1\right).\end{align*}
That is,
\begin{align}\notag
&\frac{-d+1}{\alpha}+d-\alpha-1+(\beta-1)\left(\frac{\beta-1}{\alpha}-1\right) \\ \label{b} < &\left[ \frac{2(\beta-1)}{\alpha}+1-\frac{d-\beta}{\alpha}-\frac{\beta-1}{d-\alpha-1}\left(\frac{\beta-1}{\alpha}-1\right)\right]b.
\end{align}
Since $2\beta-\alpha\le d \le 2\beta\le \alpha+\frac{5}{2}\beta-\frac{3}{2}$, we have\begin{align*}
&\frac{2(\beta-1)}{\alpha}+1-\frac{d-\beta}{\alpha}-\frac{(\beta-1)}{d-\alpha-1}\left(\frac{\beta-1}{\alpha}-1\right)\\\ge& \frac{2(\beta-1)}{\alpha} -\frac{3(\beta-1)}{2\alpha}-\frac{\beta-1}{2(\beta-\alpha-1)}\left(\frac{\beta-1}{\alpha}-1\right)=0.\end{align*}
Therefore, by the assumption that $b\le \frac{d-\alpha}{2}$, the inequality \eqref{b} yields
\begin{align*}\notag    
&\frac{-d+1}{\alpha}+d-\alpha-1+(\beta-1)\left(\frac{\beta-1}{\alpha}-1\right) \\< &\left[\frac{2(\beta-1)}{\alpha}+1-\frac{d-\beta}{\alpha}-\frac{(\beta-1)}{d-\alpha-1}\left(\frac{\beta-1}{\alpha}-1\right)\right]\frac{d-\alpha}{2} \\ \notag <& \left[\frac{2(\beta-1)}{\alpha}+1-\frac{d-\beta}{\alpha}\right]\frac{d-\alpha}{2} -\frac{(\beta-1)}{2}\left(\frac{\beta-1}{\alpha}-1\right).\end{align*}
That is, \begin{align} \label{d}
    d^2-3\beta d-\alpha^2-\alpha+3\beta^2-6\beta+5<0.
\end{align}
Since the discriminant of the above quadratic polynomial is non-negative, we deduce that $$9\beta^2\ge 4(-\alpha^2-\alpha+3\beta^2-6\beta+5).$$
Thus, we arrive at $$\beta\le \sqrt{\frac{1}{3}(2\alpha+1)^2+9}+4<\frac{2\sqrt{3}}{3}\alpha+7.$$
By \eqref{d}, we have
$$\left| d- \frac{3}{2}\beta\right|< \frac{1}{2} \sqrt{4\alpha^2-3\beta^2+4\alpha+24\beta-20})< \frac{1}{2} \sqrt{4\alpha^2-3\beta^2+28\beta}).$$
Therefore, neither condition (2) nor (3) holds, which is a contradiction.

Now, we have completed Case 2 and established Theorem \ref{thm2}.\end{proof}




\begin{remark}\label{rmk4.2}
In addition to the cases in Theorem \ref{thm2} and \cite[Theorem 1.4]{CLZ24}, our method can deal with many other cases.
Here, we comment on some possible ways to extend our conclusion.
\begin{itemize}
    \item [(i)] The condition $n<3d-2\alpha$ is not always necessary to apply our method. In particular, one can employ \eqref{m2} instead of \eqref{M2}.
    \item[(ii)] Suppose that there are $e$ edges within $S:=\{v_1,\ldots,v_b\}$ and $f$ edges between $S$ and $N_x\backslash S$. Note that $\sum_{k=1}^{b}D_k=2e$ and $\sum_{k=b+1}^{d-\alpha-1}D_k=f$. The inequality \eqref{M4} can be improved as\begin{align}
 \notag  M_4=\sum_{k=1}^{d-\alpha-1} \binom{D_k}{2}\ge \frac{1}{2} \biggl( \frac{(2e)^2}{b}+\frac{f^2}{d-\alpha-1-b}-\sum_{k=1}^{d-\alpha-1} D_k\biggr),
\end{align}
and the inequality \eqref{leb} can be replaced by $$M_1+M_2+M_3+M_4+\binom{b}{2}=\alpha e+\left(\binom{b}{2}-e\right)\beta.$$
These modifications might improve our conclusion to some extent.
\end{itemize}
\end{remark}

To illustrate the above remark (i), we show that the Lin--Lu--Yau curvature of a particular strongly regular graph with $n\geq 3d-2\alpha$ achieves the upper bound $\frac{2+\alpha}{d}$.

\begin{example}\label{ex:324}
    Consider a strongly regular graph $G$ with parameters (324, 152, 70, 72) derived from regular symmetric Hadamard matrices with
 constant diagonal (RSHCD), which is constructed in \cite[Section 5.3.2]{CP17} based on \cite[Lemma 11]{HX10}. Note that the condition $n<3d-2\alpha$ does not hold. Inserting the parameters into \eqref{M1}, \eqref{m2}, \eqref{M3},  \eqref{M4}, and \eqref{leb}, we have$$\left(\frac{107}{5670}+\frac{1}{2(b-1)}\right)X^2-\frac{5462b}{2835}X+\frac{104791b^2}{2835}-40b\le 0,$$where $2\le b\le \frac{d-\alpha}{2}=41$. Considering the discriminant, we have$$\left(\frac{5462b}{2835}\right)^2\ge 4\left(\frac{107}{5670}+\frac{1}{b-1}\right)\left(\frac{104791b^2}{2835}-40b\right).$$
 It is direct to check that the inequality above does not hold for $2\le b\le 41$. Therefore, the Lin--Lu--Yau curvature of $G$ is \[\frac{2+\alpha}{d}=\frac{9}{19}.\]
\end{example}

Now, let us prove Theorem \ref{Bonini}. Recall that, in the setting of Theorem \ref{Bonini}, we have $(n,d,\alpha,\beta)=(4\gamma+1,2\gamma,\gamma-1,\gamma)$. By the condition $(1)$ of Theorem \ref{thm2}, we only need to verify the cases that $\gamma$ is equal to $2$, $3$, $4$, $5$, or $6$.

\begin{proof}[Proof of Theorem \ref{Bonini}]
We use the framework of the proof of Theorem \ref{thm2}. It is sufficient to derive contradiction assuming $2\le b \le \frac{d-\alpha}{2}=\frac{\gamma+1}{2}$, which implies $\gamma\ge3$. By \eqref{4b}, we deduce directly that $b> \gamma-3$.
   
   Suppose that $\gamma=3$, then $b=2$.  That is, $\{ v_1, v_2 \}$ have only $1$ neighbor (denoted by $u_1$) in $N_y$. We can assume both $v_1$ and $v_2$ are adjacent to $u_1$. Since, otherwise, one of them has no neighbors in $N_y$. This reduces to the Case $b=1$, which leads to a contradiction as in the proof of Theorem \ref{thm2}. 
   Then, by Lemma \ref{lemma:neighbors}, both $v_1$ and $v_2$ have $2$ neighbors in $P_{xy}$. Since $P_{xy}$ only has $3$ elements, $v_1$ and $v_2$ have at least $1$ common neighbor in $P_{xy}$. Since $x$ and $u_1$ are also common neighbors of $v_1$ and $v_2$, it follows that $v_1$ and $v_2$ have at least $3$ common neighbors. Thus, $v_1$ is non-adjacent to $v_2$. Since both $v_1$ and $v_2$ have a neighbor in $N_x$, they are both adjacent to $v_3:=N_x\setminus\{v_1, v_2\}$. Hence, $v_1$ and $v_2$ have $4$ common neighbors, which contradicts to $\beta=\gamma=3$.

Suppose that $\gamma=4$, then $b=2$. That is, $\{ v_1, v_2 \}$ have only $1$ neighbor in $N_y$. Clearly, $X=2$. By \eqref{M1}, \eqref{M2} and \eqref{M3}, we have $M_1\ge1$, $M_2\ge2$ and $M_3\ge1$. Hence,
$$M_1+M_2+M_3+M_4+{\binom{b}{2}}>{\binom{b}{2}}\gamma,$$
which is a contradiction.

   Suppose that $\gamma=5$, then $b=3$. Substituting $\gamma=5$ and $b=3$ into the inequality \eqref{dxs2}, we have$$30-12X+\frac{23X^2}{20}\le0.$$
It follows that $X=5$ or $X=6$. 

In case that $X=5$, we derive from \eqref{M1}, \eqref{M2} and \eqref{M3} that $M_1\ge3$, $M_2\ge5$ and $M_3\ge4$ respectively. We claim that $M_4\ge1$. Set $N_x=\{v_1,\ldots,v_5\}$. Denote by $e$ the number of edges within $\{v_1,\ldots,v_5\}$ and $f$ edges between $\{v_1,\ldots,v_5\}$ and $N_x\backslash \{v_1,\ldots,v_5\}$.
Observe that $X=2e+f$ by Lemma \ref{lemma:neighbors} (ii). Thus, the number of edges within $\{v_1, v_2, v_3\}$ is at most $2$. If there are 2 edges within $\{v_1, v_2, v_3\}$, then $M_4\ge1$, as claimed. If the number of edges within $\{v_1, v_2, v_3\}$ is at most $1$, then there are at least $3$ edges between $\{v_1, v_2, v_3\}$ and $\{v_4, v_5\}$, which implies that either $v_4$ or $v_5$ is adjacent to two vertices of $\{v_1, v_2, v_3\}$. That is, we have $M_4\ge1$ as claimed.

In case that $X=6$, we derive from \eqref{M1}, \eqref{M2}, \eqref{M3} and \eqref{M4} that $M_1\ge2$, $M_2\ge4$, $M_3\ge6$ and $M_4\ge1$ respectively. 

Therefore, in either of the two cases $X=5$ and $x=6$, we have
$$M_1+M_2+M_3+M_4+{\binom{b}{2}}>{\binom{b}{2}}\gamma,$$
which is a contradiction.

Suppose that $\gamma=6$, then $b\le\frac{7}{2}$, which is contradictory to $b> \gamma-3$. 
\end{proof}

 To conclude this section, we prove the number-theoretic Corollary \ref{cor:number_theoretic}.   
\begin{proof}[Proof of Corollary \ref{cor:number_theoretic}]
Consider the corresponding Paley graph $P(n)$. By Theorem \ref{thm:upperk} and Theorem \ref{Bonini}, there is a perfect matching $\mathcal{M}$ between $N_x$ and $N_y$. Since $|S|\ge\frac{3}{4}(n-1)$, we have $\left|\left(GF(n)\setminus\{x,y\}\right)\setminus S\right|\leq \frac{1}{4}(n-1)-1$.
Recall that a Paley graph $P(n)$ with $n=4\gamma+1$ is a strongly regular graph with parameters $(4\gamma+1, 2\gamma,\gamma-1, \gamma)$. Therefore, we have $|N_x|=\gamma=\frac{1}{4}(n-1)$. That is, we have $\left|\left(GF(n)\setminus\{x,y\}\right)\setminus S\right|\leq |N_x|-1$.
Then, it is clear that there exists an edge $wz\in \mathcal{M}$, which is not adjacent to any vertex in $GF(n)\setminus S$, and the desired result follows.
\end{proof}

\section*{Acknowledgement}
This work is supported by the National Key R \& D Program of China 2023YFA1010200 and the National Natural Science Foundation of China No. 12031017 and No. 12431004. H. Z. is very grateful to Professor Wen Huang for his encouragement and helpful discussion.


\begin{thebibliography}{99}
    \bibitem{AS16} N. Alon, J. H. Spencer, The probabilistic method, 4th ed., John Wiley \& Sons, 2016.
    \bibitem{BJL12} F. Bauer, J. Jost, S. Liu, Ollivier-Ricci curvature and the spectrum of the normalized graph Laplace operator, Math. Res. Lett. 19 (2012), no. 6, 1185-1205.
    \bibitem{BM15} B. B. Bhattacharya, S. Mukherjee, Exact and asymptotic results on coarse Ricci curvature of graphs, Discrete Math. 338 (2015), no. 1, 23-42.
    \bibitem{BM08} J. A. Bondy, U. S. R. Murty, Graph theory, Springer Verlag, 2008.
    \bibitem{BCDDFP20} V. Bonini, C. Carroll, U. Dinh, S. Dye, J. Frederick, E. Pearse, Condensed Ricci curvature of complete and strongly regular graphs, Involve, 13 (2020), no. 4, 559-576.
    \bibitem{BCCST24} V. Bonini, D. Chamberlin, S. Cook, P. Seetharaman, T. Tran, Condensed Ricci curvature on Paley graphs and their generalizations, arXiv: 2409.03631, 2024. 
    \bibitem{BCLMP18} D. Bourne, D. Cushing, S. Liu, F. M\"unch, N. Peyerimhoff, Ollivier-Ricci idleness functions of graphs, SIAM J. Discrete Math. 32 (2018), no. 2, 1408-1424.
    \bibitem{BCN89} A. E. Brouwer, A. M. Cohen, A. Neumaier, Distance-regular graphs, Springer Verlag, 1989.
    \bibitem{CLZ24} K. Chen, S. Liu, H. Zhang, Curvature and local matchings of conference graphs and extensions, arXiv: 2409.06418v2, 2024.
    \bibitem{CY96} F. R. K. Chung, S.-T. Yau, Logarithmic Harnack inequalities, Math. Res. Lett. 3 (1996), no. 6, 793-812.
    \bibitem{CP17} N. Cohen, D. V. Pasechnik, Implementing Brouwer's database of strongly regular graphs, Des. Codes Cryptogr. 84 (2017), 223-235.
    \bibitem{CKKLMP23} D. Cushing, S. Kamtue, R. Kangaslampi, S. Liu, F. M\"unch, N. Peyerimhoff, Bakry--\'Emery and Ollivier Ricci curvature of Cayley graphs, Electron. J. Comb. 32 (2025), no. 3, P3.21.
    \bibitem{CKKLP20} D. Cushing, S. Kamtue, R. Kangaslampi, S. Liu, N. Peyerimhoff, Curvatures, graph products and Ricci flatness, J. Graph Theory 96 (2021), no. 4, 522-553.
    \bibitem{CKKLMP20} D. Cushing, S. Kamtue, J. Koolen, S. Liu, F. M\"unch, N. Peyerimhoff, Rigidity of the Bonnet-Myers inequality for graphs with respect to Ollivier Ricci curvature, Adv. Math. 369 (2020), 107188.
   \bibitem{DOJ19} M. Da\v gl\i, O. Olmez, J. D. H. Smith, Ricci curvature, circulants, and extended matching conditions, Bull. Korean Math. Soc. 56 (2019), no. 1, 201-217.
   \bibitem{HX10} W. H. Haemers, Q. Xiang, Strongly regular graphs with parameters $(4m^4,2m^4+m^2,m^4+m^2,m^4+m^2)$ exist for all $m>1$, Eur. J. Comb. 31 (2010), no. 6, 1553-1559.
   \bibitem{HLX24} X. Huang, S. Liu, Q. Xia, Bounding the diameter and eigenvalues of amply regular graphs via Lin--Lu--Yau curvature, Combinatorica 44 (2024), no. 6, 1177-1192.
   \bibitem{IS20} E. Iceland, A. Samorodnitsky, On coset leader graphs of structured linear codes, Discrete Comput. Geom. 63 (2020), no. 6, 560-576.
   \bibitem{JL14} J. Jost, S. Liu, Ollivier's Ricci curvature, local clustering and curvature-dimansioin inequalities on graphs, Discrete Comput. Geom. 51 (2014), no. 2, 300-322.
   \bibitem{LL21} X. Li, S. Liu, Lin--Lu--Yau curvature and diameter of amply regular graphs, Journal of University of Science and Technology of China (JUSTC) 51 (2021), no. 12, 889-893.
   \bibitem{LLY11} Y. Lin, L. Lu, S.-T. Yau,
   Ricci curvature of graphs,
   Tohoku Math. J. (2) 63 (2011), no. 4, 605-627.
   \bibitem{LY10} Y. Lin, S.-T. Yau, Ricci curvature and eigenvalue estimate on locally finite graphs,
   Math. Res. Lett. 17 (2010), no. 2, 343-356.
   \bibitem{M23} F. M\"unch, Non-negative Ollivier curvature on graphs, reverse Poincaré inequality, Buser inequality, Liouville property, Harnack inequality and eigenvalue estimates, J. Math. Pures Appl. 170 (2023), no. 9, 231-257.
   \bibitem{MS23} F. M\"unch, J. Salez, Mixing time and expansion of non-negatively curved Markov chains, J. \'Ec. polytech. Math. 10 (2023), 575-590.
   \bibitem{MW19} F. M\"unch, R. Wojciechowski, Ollivier Ricci curvature for general graph Laplacians: Heat equation, Laplacian comparison, non-explosion and diameter bounds, Adv. Math. 356 (2019), 106759.
   \bibitem{Ollivier09} Y. Ollivier, Ricci curvature of Markov chains on metric spaces, J. Funct. Anal. 256 (2009), no. 3, 810-864.
   \bibitem{Ollivier10} Y. Ollivier, A survey of Ricci curvature for metric spaces and Markov chains, Probabilistic Approach to Geometry, Adv. Stud. Pur Math., M. Kotani, M. Hino, T. Kumagai (eds.), Tokyo, Math. Soc. Japan, 2010, 343-381.
   \bibitem{Salez} J. Salez, Sparse expanders have negative curvature, Geom. Funct. Anal. 32 (2022), no. 6, 1486-1513.
   \bibitem{Smith} J. D. H. Smith, Ricci curvature, circulants, and a matching condition, Discrete Math. 329 (2014), 88-98.

\end{thebibliography}
\end{document}